\documentclass[10pt,a4paper]{amsart}
\usepackage[latin1]{inputenc}
\usepackage{amsmath,amsfonts,amssymb}
\usepackage{enumerate}
\usepackage{graphicx}
\usepackage[colorlinks=true]{hyperref}
\numberwithin{equation}{section}
\usepackage{caption}

\author{Philippe Nadeau}
\thanks{The author is supported partly by the Austrian Science Foundation (FWF), in the framework of the Wittgensteinpreis (grant Z130-N13) and the National Research Network "Analytic Combinatorics and Probabilistic Number Theory" (grant S9607-N13)}
\address{Fakult\"at f\"ur Mathematik, Universit\"at Wien, Garnisonga{\ss}e 3, A-1090 WIEN, AUSTRIA. }
\email{philippe.nadeau@univie.ac.at}
\date{}

\newtheorem{thm}{Theorem}
\newtheorem{prop}[thm]{Proposition}
\newtheorem{lemma}[thm]{Lemma}

\theoremstyle{definition}
\newtheorem{defi}[thm]{Definition}

\newcommand{\dem}{\noindent \textbf{Proof: }}
\newcommand{\demof}[1]{\noindent \textbf{Proof of~{#1}:}}
\newcommand{\findem}{\vspace{-.55cm} \begin{flushright} $\square~$
\end{flushright} \vspace{.2cm} }

\newcommand{\si}{\sigma}
\newcommand{\lsi}{{\lambda(\sigma)}}
\newcommand{\lsis}{{\lambda(\sigma^*)}}
\newcommand{\ltaus}{{\lambda(\tau^*)}}
\newcommand{\lpi}{{\lambda(\pi)}}
\newcommand{\ltau}{{\lambda(\tau)}}

\newcommand{\tspt}{{t_{\si,\tau}^\pi}}
\newcommand{\Tspt}{{\mathbf{TFPL}_{\si,\tau}^\pi}}
\newcommand{\oTspt}{{\overrightarrow{\mathbf{TFPL}}_{\si,\tau}^\pi}}
\newcommand{\otspt}{{\overrightarrow{t}_{\si,\tau}^\pi}}
\newcommand{\oPhi}{\overrightarrow{\Phi}}

\newcommand{\ssyt}{\operatorname{SSYT}}

\newcommand{\Dn}{\mathcal{D}_n}

\newcommand{\size}{\operatorname{d}}
\newcommand{\can}{\overrightarrow{\operatorname{can}}}

\newcommand{\cspt}{\mathbf{c}_{\lsi,\ltau}^{\lpi}}
\newcommand{\KT}{\mathbf{K}\mathbf{T}_{\si,\tau}^\pi}

\title[TFPL and LR coefficients]{Fully Packed Loop configurations in a Triangle and Littlewood--Richardson coefficients.}

\begin{document}
\begin{abstract}
  In this work we continue our study of Fully Packed Loop (FPL) configurations in a triangle. These are certain subgraphs on a triangular subset of $\mathbb{Z}^2$, which first arose in the study of the usual FPL configurations on a square grid.  We show that, in a special case, the enumeration of these FPLs in a triangle is given by Littlewood--Richardson coefficients. The proof consists of a bijection with Knutson--Tao puzzles.
\end{abstract}

\maketitle


\section*{Introduction}
Fully Packed Loop configurations (FPLs) are configurations on a square grid to which a certain link pattern $\pi$ is attached (see Section~\ref{sub:fpl} for definitions). As was first shown in~\cite{CKLN,Thapper}, 
FPL configurations in a Triangle (TFPLs) occur naturally when studying certain properties of the usual FPLs. The connection is explained in more detail in~\cite{TFPL1}, where in several properties of TFPLs are also shown. The study of TFPLs was initially motivated by the desire to obtain an expression for the number $A_\pi$ of FPLs with link pattern $\pi$; also, as conjectured in~\cite{Thapper} and proved in~\cite{TFPL1}, there exist linear relations between FPLs whose coefficients are defined in terms of TFPLs. These relations are in fact recurrence formulas allowing to compute all ${A}_\pi$. The famous Razumov-Stroganov correspondence \cite{ProofRS,RS-conj} gives rise to a different set of relations characterizing the $A_\pi$; they concern FPLs all on the same grid size, while the recurrence formulas express FPLs of size $n$ in terms of FPLs of size $n-1$.

Boundaries of TFPLs can be encoded by Dyck words $\si,\tau,\pi$, to which partitions $\lsi,\ltau,\lpi$ are attached; our main result (Theorem~\ref{th:main}) is:
\smallskip

\emph{\indent If $|\lsi|+|\ltau|=|\lpi|$, then the number of TFPL configurations with boundary $\si,\tau,\pi$ is given by the Littlewood-Richardson coefficient $\cspt$.}
\smallskip

 Relevant definitions can be found in Section~\ref{sec:prelim}. It is quite surprising to come across such coefficients when starting with a purely enumerative problem: objects counted by Littlewood--Richardson coefficients do not seem to appear ``by chance'', as they are usually designed to give a combinatorial interpretation of these numbers. Such is the case of Knutson--Tao puzzles, which are the objects we shall put into one-to-one correspondence with our TFPL configurations. We will see that although the map itself is quite easy to define, proving that it is indeed a bijection requires some work.

Let us give a brief outline: in Section~\ref{sec:prelim} we define FPLs both on a grid and in a triangle, and state Theorem~\ref{th:main} which is our main result. We also explain briefly the connection between FPLs and TFPLs, which gives rise to Formula~\eqref{eq:api}. Section~\ref{sec:ingredients} introduces three major ingredients in the proof of the main theorem, which is itself proved in Section~\ref{sec:bij}.
\smallskip

\emph{Note: this work was first presented at the Fpsac 2010 conference in San Francisco~\cite{NadFPLFpsac}.}


\section{Preliminaries}
\label{sec:prelim}

\subsection{Fully Packed Loop configurations and link patterns}
\label{sub:fpl}
We fix a positive integer $n$, and let $G_n$ be the square grid with $n^2$ vertices. We impose {\em periodic boundary conditions} on $G_n$, which means that we select every other external edge on the grid, starting by convention with the topmost on the left side; we number these $2n$  external edges counterclockwise. A {\em Fully Packed Loop (FPL)} configuration $F$ of size $n$ is a subgraph of $G_n$  such that each vertex of $G_n$ is incident to two edges of $F$. An example of FPL configuration is given on Figure~\ref{fig:fplexample} (left).

\begin{figure}[!ht]
\begin{center}
\includegraphics[width=0.8\textwidth]{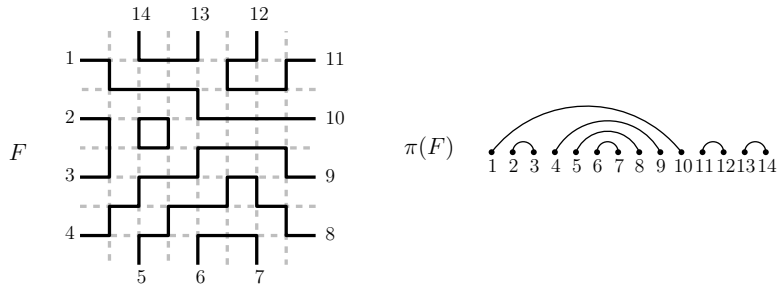}
\caption{A FPL configuration with its associated link pattern.
\label{fig:fplexample}}
\end{center}
\end{figure}

Define a {\em link pattern} $\pi$ of size $n$ as a partition of $\{1,\ldots,2n\}$ in $n$ {\em pairwise noncrossing} pairs $\{i,j\}$, which means that there are no integers $i<j<k<\ell$ such that $\{i,k\}$ and $\{j,\ell\}$ are both in $\pi$. A FPL configuration $F$ on $G_n$ naturally defines non-crossing paths between its external edges, so we can define the link pattern $\pi(F)$ as the set of pairs $\{i,j\}$ where $i,j$ label external edges which are the extremities of the same path in $F$. This is illustrated on the right of Figure~\ref{fig:fplexample}. If $\pi$ is a link pattern, we denote by ${A}_\pi$ the number of $FPL$ configurations $F$ of size $n$ such that $\pi(F)=\pi$.
 

\subsection{Words and Ferrers diagrams}
\label{sub:words}
 A link pattern $\pi$ on $\{1,\ldots,2n\}$ can be encoded by a binary word of length $2n$, where for each pair $\{i<j\}$ in $\pi$ we set $\pi_i=0$ and $\pi_j=1$. Such words $\pi$ form the following subset of $\{0,1\}^{2n}$: 
\begin{defi}[$\Dn$]
We denote by $\mathcal{D}_n$ the set of words $\si$ of length $2n$, such that $|\si|_0=|\si|_1=n$, and each prefix $u$ of $\si$ verifies $|u|_0\geq|u|_1$.
\end{defi}
 These are known as Dyck words, and are counted by the Catalan numbers $\frac{1}{n+1}\binom{2n}{n}$. Note also that if $w\in\Dn$, then it contains a factor $01$, and the removal of such a factor leaves a shorter Dyck word in $\mathcal{D}_{n-1}$.

We will identify link patterns with words in $\Dn$. There is also a bijection $\si\mapsto \lsi$ from $\Dn$ to the set of Ferrers diagrams included in the staircase shape $\delta_n=(n-1,n-2,\ldots,1,0)$: the letters of $\si$ encode the Southeast boundary of the Ferrers shape from bottom to top, cf. Figure~\ref{fig:Dn}. 

\begin{figure}[!ht]
\begin{center}
\includegraphics[width=0.8\textwidth]{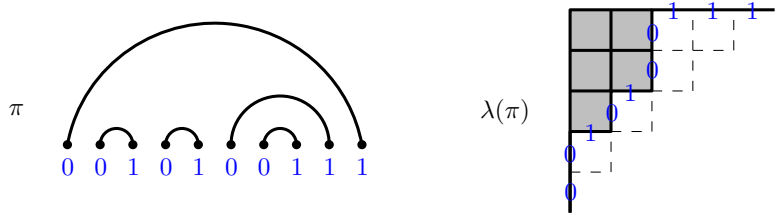}
\caption{The word $0010100111\in \mathcal{D}_5$ as a link pattern and a Ferrers diagram.
\label{fig:Dn}}
\end{center}
\end{figure}

The {\em degree} $\size(\si)$ of $\si$ is the number of indices $i<j$ such that $(\si_i,\si_j)=(1,0)$, and is equal to the number of boxes $|\lsi|$. For instance we have $\size(\pi)=5$ for the example of Figure~\ref{fig:Dn}. The \emph{conjugate} $\si^*$ of $\si=\si_1\cdots\si_{n}$ is the word of length $n$ defined by $\si^*_i:=1-\si_{n+1-i}$; the diagram $\lsis$ is obtained from $\lsi$ by a reflection about the main diagonal of $\delta_n$. We define the partial order $\si\leq \tau$ if $\lsi\subseteq\ltau$ in the diagram representation. 

A {\em semistandard Young tableau} of shape $\lambda$  is a filling of the boxes of $\lambda$ by positive integers, in which entries are nondecreasing across each row from left to right and increasing down each column. We denote by $\ssyt(\lambda,N)$ the number of such tableaux where entries are in $\{1,\ldots,N\}$.

Let $u$ be a box in $\lambda$ which is in the $k$th row from the top and $\ell$th column from the left. The {\em content} $c(u)$ of $u$ is defined as $\ell-k$; its {\em hook-length} $h(u)$ is the number of boxes in $\lambda$ which are below $u$ and in the same column, or right of $u$ and in the same row ($u$ itself being counted just once). Define then $H_\lambda=\prod h(u)$ where the product is over all cells $u$ of the diagram $\lambda$. 

We have then the {\em hook content formula}:  
 \begin{equation}
\label{eq:hookcontent}
\ssyt(\lambda,N)= \frac{1}{H_\lambda}\prod_{u\in\lambda}(N+c(u)),
\end{equation}
so that $\ssyt(\lambda,N)$ is given by a polynomial in $N$ with leading term $\frac{1}{H_\lambda}N^{|\lambda|}$. For the diagram in Figure~\ref{fig:Dn} we get $\frac{1}{24}(N+1)N^2(N-1)(N-2)$.

\subsection{Fully packed Loops in a triangle}
\label{sub:tfpl}

We define the triangle $\mathcal{T}^n\subseteq \mathbb{Z}^2$ as the set of points $(x,y)$ which verify $x\geq y \geq 0$ and $x+y\leq 4n-2$. We also include the following edges in the definition of $\mathcal{T}^n$: $2n$ vertical edges $(e_i)_i$ below the vertices $(2i-2,0)$ for $i=1\ldots 2n$; $2n-1$ horizontal edges between $(i,i)$ and $(i+1,i)$, as well as  between $(4n-2-i-1,i)$ and $(4n-2-i,i)$, for $i=0,\ldots,2n-2$. These edges are in bold on Figure~\ref{fig:tfpl}, left. A vertex of $\mathcal{T}^n$ is {\em even} or {\em odd} depending on the parity of its sum of coordinates: we represent even vertices by filled squares and odd ones by squares with a white interior. We also call \emph{inner vertices} the vertices of $\mathcal{T}^n$ which are not of the form $(i-1,i-1)$ (\emph{left vertices}) or $(2n-2+i,2n-i)$ (\emph{right vertices}).

We now impose extra conditions given by $\si,\tau$ words in $\Dn$. If $\si=\si_1\ldots\si_{2n}$, we add a vertical edge below $(i-1,i-1)$ for each $i$ such that $\si_i=0$, and forbid such an edge if $\si_i=1$. If $\tau=\tau_1\ldots\tau_{2n}$, we add a vertical edge below $(2n-2+i,2n-i)$ for each $i$ such that $\tau_i=1$, and forbid such an edge if $\tau_i=0$. Let $\mathcal{T}^n(\si,\tau)$ be the corresponding triangle.

\begin{figure}[!ht]
\begin{center}
\includegraphics[width=\textwidth]{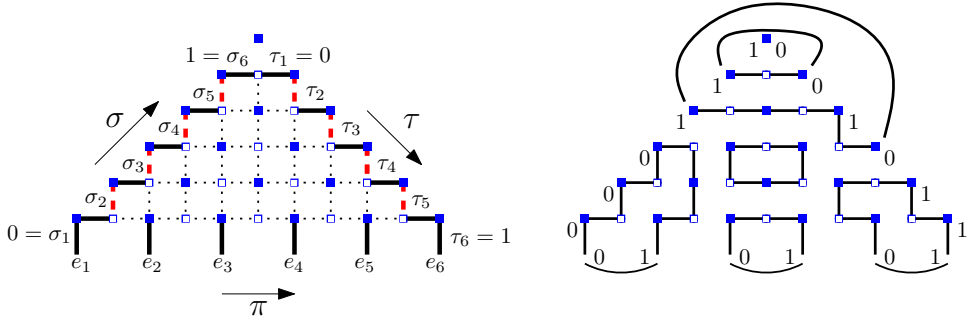}
\caption{Triangle $\mathcal{T}^n(\si,\tau)$ and an example of TFPL.
\label{fig:tfpl}}
\end{center}
\end{figure}

\begin{defi}[TFPLs]
\label{defi:tfpls}
A \emph{FPL configuration $f$ in a triangle (TFPL)} with boundary conditions $\si,\tau,\pi$ in $\Dn$ is a graph on $\mathcal{T}^n(\si,\tau)$, such that all inner vertices  are imposed to be of degree $2$, and furthermore (I) for each pair $\{i,j\}$ in $\pi$, the edges $e_i$ and $e_j$ are linked by a path in $\mathcal{T}^n$ , and (II) the paths starting from a left (\emph{resp.} right) vertex must end at a right (\emph{resp.} left) vertex.
\end{defi} 

An example is shown on Figure~\ref{fig:tfpl}, right. The set of these TFPLs is denoted by $\Tspt$, and we let $\tspt$ be its cardinality. We have clearly a \emph{left-right symmetry} in TFPLs: the reflection of a TFPL configuration about a vertical axis is a TFPL configuration. If $f$ belongs to $\Tspt$, then its reflection belongs to $\mathbf{TFPL}_{\tau^*,\si^*}^{\pi^*}$.
Note that because of condition (II) there are three kinds of paths in a TFPL: \emph{Left-Right paths} which join a left vertex to a right vertex, \emph{Bottom paths} which link two edges $e_i$ and $e_j$, and \emph{Closed paths}.

The main result of this paper is the following enumeration, in which $\cspt$ is a Littlewood--Richardson coefficient, cf. Section~\ref{sub:puzzles}:

\begin{thm}
\label{th:main}
Suppose $\si,\tau,\pi\in\Dn$ verify $\size(\si)+\size(\tau)=\size(\pi)$. Then we have
 \[
  \tspt=\cspt.
 \]

\end{thm}
 
We will prove this in Section~\ref{sec:bij}; in Section~\ref{sec:ingredients} we give the key tools of this proof. Let us note that Theorem~\ref{th:main} was independently conjectured in~\cite[Lemma 2 and following remark]{ZJtriangle}, as a consequence of the main conjecture of that paper which is a formula for $\tspt$ \emph{for any $\si,\tau,\pi\in\Dn$}.
\medskip

\subsection{Link between FPLs and TFPLs} 
Given a link pattern $\pi$ on $\{1,\ldots,2n\}$, and an integer $m\geq 0$, define $\pi\cup m$ as the link pattern on $\{ 1,\ldots,2(n+m)\}$ given by the nested pairs $\{i,2n+2m+1-i\}$ for $i=1\ldots m$, and the pairs $\{i+m,j+m\}$ for each $\{i,j\}\in \pi$. As words, note that this is simply $\pi\cup m=0^m\pi 1^m $. We introduce the notation $\displaystyle{A_\pi(m):=A_{\pi\cup m}}$, so that in particular $A_\pi(0)=A_\pi$. The structure of TFPLs arises naturally when considering FPL configurations enumerated by $A_\pi(m)$ for $m$ big enough, see~\cite{CKLN}. It was shown in~\cite{CKLN,Thapper,TFPL1} that we have the following formula: for $m,k$ integers,
\begin{equation}
\label{eq:api}
A_\pi(m)=\sum_{\si,\tau\in \Dn}\ssyt(\lsi,n+k)\cdot \tspt\cdot  \ssyt(\ltaus,m-k-2n+1),
\end{equation}

Using this expression and certain properties of TFPLs, one can show in particular the following result, which was conjectured in~\cite{Zuber-conj}:

\begin{thm}[\cite{CKLN}]
\label{th:CKLN}
 $A_\pi(m)$ is a polynomial in $m$ with leading term $\frac{1}{H(\pi)}m^{\size(\pi)}$.
\end{thm}

From Equations~\eqref{eq:hookcontent} and\eqref{eq:api}, we see that an expression for the numbers $\tspt$ will lead to a formula for $A_\pi(m)$. What Theorem~\ref{th:main} shows is that the numbers $\tspt$ are also interesting on their own.

\section{Ingredients of the proof}
\label{sec:ingredients}

In Section~\ref{sub:oTFPL} we introduce certain variations on TFPL configurations called \emph{oriented} TFPL configurations; whereas the former use connectivity conditions, which are {\em global}, the latter only necessitate \emph{local} conditions which make them easier to manipulate. Then we introduce Knutson--Tao puzzles in Section~\ref{sub:puzzles}, which are combinatorial objects counted by Littlewood--Richardson numbers; the proof of Theorem~\ref{th:main} will consist of a bijection between these puzzles and TFPLs. The correctness of this bijection will make use in an essential way of the identity~\eqref{eq:identity_tc} proved in Section~\ref{sub:summation}.

\subsection{Oriented TFPL configurations}
\label{sub:oTFPL}
We define $\mathcal{T}_{\rightarrow}^n(\si,\tau)$ as the graph $\mathcal{T}^n(\si,\tau)$ with all edges on the left side oriented up and right, and all edges on the right side oriented down and right; see example on Figure~\ref{fig:otfpls}, left.

\begin{defi}[oriented TFPLs]
\label{defi:otfpls}
An \emph{oriented TFPL configuration $f$} with boundary conditions $\si,\tau,\pi$ in $\Dn$ is a directed graph on $\mathcal{T}_{\rightarrow}^n(\si,\tau)$ such that all inner vertices have exactly one incoming edge and one outgoing edge, and in which the edge $e_i$ is directed upwards (\emph{resp.} downwards) if $\pi_i=0$ (\emph{resp.} $\pi_i=1$).
\end{defi} 

\begin{figure}[!ht]
\begin{center}
\includegraphics[width=\textwidth]{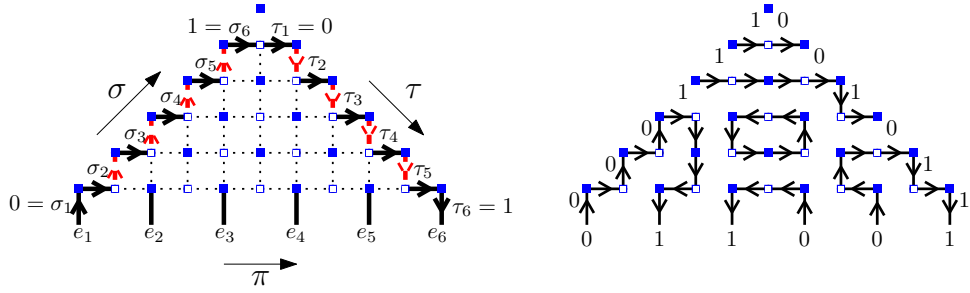}
\caption{Triangle $\mathcal{T}_{\rightarrow}^n(\si,\tau)$ and an example of oriented TFPL.
\label{fig:otfpls}}
\end{center}
\end{figure}

The set of oriented TFPLs is denoted by $\oTspt$ and its cardinality $\otspt$. Like TFPLs, oriented TFPLs also possess a \emph{left-right symmetry}: by reflecting the underlying non-oriented graph of an oriented TFPL  about a vertical axis, and inverting the orientation of up and down edges while keeping left and right orientations, one obtains another oriented TFPL configuration. This is easily seen to be an involution exchanging  $\oTspt$ and $\overrightarrow{\mathbf{TFPL}}_{\tau^*,\si^*}^{\pi^*}$.

\begin{prop}
\label{prop:oTFPLs_to_TFPLs}  Let us be given an oriented TFPL configuration $F\in\oTspt$.
\begin{enumerate}[(a)]
\item Its underlying (non-oriented) graph is a TFPL configuration $f$.\label{ita}
\item The directed edges induce a global orientation on each path of $f$, where all Left-Right paths are oriented from left to right. \label{itb}
\item The TFPL $f$ belongs to $\mathcal{T}_{\si,\tau}^{\pi'}$ for a certain link pattern $\pi'$, and $\pi'=\pi$ if and only all Bottom paths are oriented from left to right.\label{itc}
\end{enumerate}
\end{prop}

\dem Given an oriented configuration, it is clear that the underlying graph is such that each inner vertex has degree $2$, and the boundary conditions are the same as those of TFPLs. It is also clear that the directed edges induce a global orientation on paths, and that indeed Left-Right paths oriented from left to right by the definition of $\mathcal{T}_{\rightarrow}^n(\si,\tau)$, so~\eqref{itb} is clear.

To prove~\eqref{ita}, it remains to prove condition (II) in Definition~\ref{defi:tfpls}. We only need to show that paths starting from a left vertex necessarily end at a right vertex, the converse being then true by the  Left-Right symmetry of oriented TFPLs.

Consider a path $p$ starting at a left vertex. It cannot end at another such vertex because of the orientation of $\mathcal{T}_{\rightarrow}^n(\si,\tau)$. Suppose $p$ ends at an edge $e_i$, so that by Definition~\ref{defi:otfpls} one has $\pi_i=1$. Since $\pi$ is a Dyck path, its prefix $\pi_1\ldots\pi_{i-1}$ contains strictly more $0$s than $1$s, so that there are more up-edges than down-edges among the $i-1$ first bottom edges $\{e_1,\ldots,e_{i-1}\}$. So there is a path starting in $\{e_1,\ldots,e_{i-1}\}$ that does not end there. This is a contradiction, since this path cannot end at a left vertex (the orientations don't match), or anywhere else on the boundary since it would then have to cross the path $p$. Therefore all paths starting from a left vertex must end at a right vertex, and as argued above this proves~\eqref{ita}.

 So $f$ is a TFPL, and, since $\mathcal{T}_{\rightarrow}^n(\si,\tau)$ is simply a certain orientation of $\mathcal{T}^n(\si,\tau)$, we know that $f$ belongs in fact to $\mathcal{T}_{\si,\tau}^{\pi'}$ for a certain $\pi'\in\Dn$. Assume $f$ verifies the link pattern $\pi$, so that for any pair $\{i<j\}$  in $\pi$, $f$ contains a path between $e_i$ and $e_j$. Now, considering $\pi$ as a word one has $\pi_i=0$ and $\pi_j=1$ and therefore $e_i$ is an up-edge and $e_j$ is a down-edge in $F$, so that the path between them is oriented from left to right.

Conversely assume all bottom paths in $F$ are oriented from left to right, and consider a pair $(\pi_i,\pi_{i+1})=(0,1)$ which exists since $\pi\in\Dn$. This means that $F$ contains a path from $e_i$ going right, and a path arriving to $e_{i+1}$ from the left. Since paths are non-crossing, this must be the same path, and thus $f$ respects the pair $\{i,i+1\}$ of the link pattern $\pi$. Consider now the Dyck word $\pi'$, obtained by removing $\pi_i\pi_{i+1}$ from $\pi$, and the bottom edges $\{e_1,\ldots,e_{2n}\}\backslash\{e_i,e_{i+1}\}$; we can then use the same argument, and by immediate induction we find that Bottom paths in $f$ verify the pattern $\pi$, proving~\eqref{itc}.
\findem

\subsection{Littlewood--Richardson coefficients and  Knutson--Tao puzzles}
\label{sub:puzzles}
We refer to~\cite{StanleyEnum2} for background on symmetric functions. Let $\mathbf{x}=(x_1,x_2,\ldots)$ be commuting variables, and let $\Lambda(\mathbf{x})$ be the algebra of symmetric functions in $\mathbf{x}$. For $\lambda$ a Ferrers diagram, the Schur function $s_\lambda(\mathbf{x})$ can be defined by
\begin{equation}
\label{eq:def_schur}
 s_\lambda(\mathbf{x})=\sum_{T}\prod_{i\geq 1} x_i^{T_i}
\end{equation}
where $T$ runs through all semistandard tableaux of shape $\lambda$, and $T_i$ is the number of entries equal to $i$ in the tableau $T$.

 The Schur functions form a linear basis of $\Lambda(\mathbf{x})$, and the {\em Littlewood--Richardson (LR) coefficients } $\mathbf{c}^{\lambda}_{\mu,\nu}$ are the corresponding structure constants defined by
 \[
  s_\mu(\mathbf{x})s_\nu(\mathbf{x})=\sum_{\lambda}\mathbf{c}^{\lambda}_{\mu,\nu}s_{\lambda}(\mathbf{x}).
 \]
 The LR coefficient $\mathbf{c}^{\lambda}_{\mu,\nu}$ is $0$ unless $\mu\subseteq\lambda,\nu\subseteq\lambda$ and $|\mu|+|\nu|=|\lambda|$; also, they are known to be nonnegative integers by character theory~\cite[p.355]{StanleyEnum2}. Many combinatorial descriptions of them are also known, the most famous being the original Littlewood--Richardson rule~\cite{LR}. We will here use Knutson--Tao puzzles~\cite{KTpuzzles,KTW2}:
let $n$ be an integer, and $\si,\tau,\pi$ words in $\Dn$. Consider a triangle with edge size $2n$ on the regular triangular lattice, where unit edges on left, bottom and right side are labeled by $\si,\tau,\pi$ respectively when read from left to right.
\begin{defi}[Knutson--Tao puzzle]

 A {\em Knutson--Tao (KT) puzzle} with boundary $\si,\tau,\pi$ is a labeling of each internal edge of the triangle with $0$,$1$ or $2$, such that the labeling induced on each of the $(2n)^2$ unit triangles is composed either of three $0$s, or of three $1$s, or of $0,1,2$ in counterclockwise order. 
\end{defi}

\begin{figure}[!ht]
\begin{center}
\includegraphics[width=0.9\textwidth]{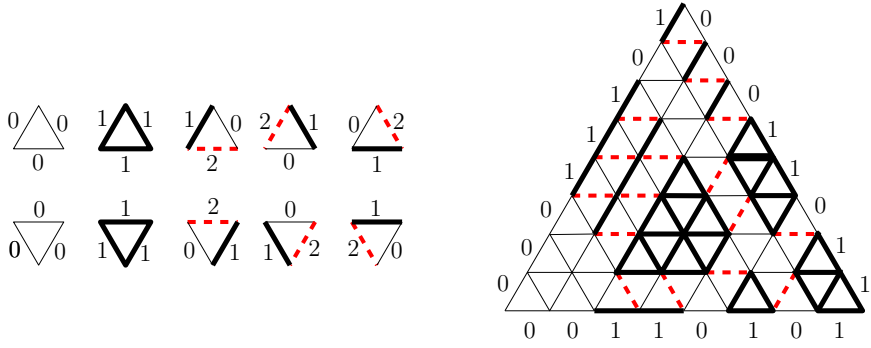}
\caption{Authorized unit triangles in a KT puzzle. \label{fig:puzzle}}
\end{center}
\end{figure}

We let $\KT$ be the set of Knutson-Tao puzzles with boundary $\si,\tau,\pi$. The list of all ten authorized labelings of unit  triangles is given on the left of Figure~\ref{fig:puzzle}, while on the right appears a puzzle with boundaries $\si=00011101$, $\tau=00011011$, $\pi=00110101$. KT puzzles give a combinatorial interpretation for LR coefficients:

\begin{thm}[\cite{KTpuzzles,KTW2}]
\label{th:KTLR}
Given $\si,\tau,\pi$ in $\Dn$, the number of KT puzzles with boundary $\si,\tau,\pi$ is given by $\cspt$. In particular, there are no such puzzles unless $\si\leq\pi,\tau\leq \pi,$ and $\size(\si)+\size(\tau)=\size(\pi)$.
\end{thm}


\subsection{A useful formula}
\label{sub:summation}
The goal of this section is to prove the following formula which holds  for any $\pi\in\Dn$:
\begin{equation}
\label{eq:identity_tc}
\mathop{\sum_{\si,\tau\in\Dn}}_{\size(\si)+\size(\tau)=\size(\pi)} \left(\tspt-\cspt\right)\cdot \frac{1}{H_{\lsi}H_{\ltau}}=0.
\end{equation}

This will be a consequence of the two formulas~\eqref{eq:identity_t} and~\eqref{eq:identity_c}.

\begin{prop}
\label{prop:thapperdegree}
One has $\tspt=0$ unless $\size(\si)+\size(\tau)\leq \size(\pi)$. Furthermore, for every $\pi\in\Dn$ we have
\begin{equation}
\label{eq:identity_t}
\frac{1}{H_\lpi}=\mathop{\sum_{\si,\tau\in\Dn}}_{\size(\si)+\size(\tau)=\size(\pi)} \tspt\cdot \frac{1}{2^{\size(\si)}H_{\lsi}}\cdot \frac{1}{2^{\size(\tau)}H_{\ltau}}.
\end{equation}
\end{prop}

\begin{proof}
We use an argument of ~\cite[Lemma 3.7]{Thapper}, in which the first part of the proposition is proved.  Assume $m$ is an even integer; Equation~\eqref{eq:api} with $k=m/2-n$ becomes 
\begin{equation}
\label{eq:1}
 A_\pi(m)=\sum_{\si,\tau\in \Dn}\ssyt(\lsi,m/2)\cdot \tspt \cdot \ssyt(\ltaus,m/2-n+1).
\end{equation}
This is a polynomial identity in $m$. By Theorem~\ref{th:CKLN}, $A_\pi(m)$ is a polynomial with leading term $\frac{1}{H_\lpi}m^{\size(\pi)}$. Therefore the coefficients of degree $>\size(\pi)$ must vanish on the r.h.s., which implies the first part of the proposition. Since a polynomial $\ssyt(\lambda,m)$ has leading term $\frac{1}{H_\lambda}m^{\size(\lambda)}$, the second part follows by taking the coefficient of degree ${\size(\pi)}$ on both sides of~\eqref{eq:1}.
\end{proof}

 From Equation~\eqref{eq:def_schur}, we have that $s_\lambda(\mathbf{x})$ specializes to  $\ssyt(\lambda,N)$ under the substitutions $x_i=1$ for $i=1\ldots N$, and $x_i=0$ otherwise. Now if one considers the Schur function $s_\lambda(\mathbf{x},\mathbf{y})$ in the variables $\{x_1,x_2,\ldots,y_1,y_2,\ldots\}$, then we have  $s_{\lambda}(\mathbf{x},\mathbf{y})=\sum_{\mu,\nu} \mathbf{c}^{\lambda}_{\mu,\nu}s_\mu (\mathbf{x})s_\nu (\mathbf{y})$ (see~\cite[p.341]{StanleyEnum2}). By specializing at $x_i=y_i=1$ for $i=1\ldots m$ and $x_i=y_i=0$ for $i>m$, we thus get the following polynomial identity in $m$:
 \[
\ssyt(\lambda,m)= \sum_{\mu,\nu} \mathbf{c}^{\lambda}_{\mu,\nu} \ssyt(\mu,m/2) \ssyt(\nu,m/2),
\] 
which in top degree becomes:
\begin{equation}
\label{eq:identity_c}
\frac{1}{H_{\lambda}}=\sum_{\mu,\nu} \mathbf{c}_{\mu,\nu}^{\lambda}\cdot \frac{1}{2^{|\mu|}H_{\mu}}\cdot \frac{1}{2^{|\nu|}H_{\nu}}.
\end{equation}

Now we use~\eqref{eq:identity_c} with $\lambda=\lambda(\pi)$ for $\pi\in\Dn$. Since $\mathbf{c}_{\mu,\nu}^{\lpi}=0$ unless $\mu,\nu\subseteq \lambda(\pi)$ and $|\mu|+|\nu|=|\lambda|$, we can replace $\mu,\nu$ in~\eqref{eq:identity_c} with $\mu=\lsi,\nu=\ltau$ and sum over words $\si,\tau\in\Dn$. We then subtract it from~\eqref{eq:identity_t}: after removing the common factor $2^{\size(\si)}2^{\size(\tau)}=2^{\size(\pi)}$, we obtain~\eqref{eq:identity_tc}. \findem

\section{The bijection and its consequences}
\label{sec:bij}

Given three words $\si,\tau,\pi\in \Dn$ verifying $\size(\si)+\size(\tau)=\size(\pi)$, we will define here a bijection between KT puzzles and TFPL configurations with the same boundary conditions $\si,\tau,\pi$. By Theorem~\ref{th:KTLR}, this will then give a bijective proof of Theorem~\ref{th:main}.

\subsection{Definition}
\label{sub:defi_phi}
Given three words $\si,\tau,\pi\in \Dn$ verifying $\size(\si)+\size(\tau)=\size(\pi)$, let $P$ be a KT puzzle with boundary $\si,\tau,\pi$. By definition each labeled unit triangle in $P$ must be labeled as one of the ten possibilities in Figure~\ref{fig:puzzle}. Now to each of these triangles we apply the transformations described in Figure~\ref{fig:rulesoriented}, and delete the original puzzle. Some cosmetic modifications are necessary: rescale the resulting configuration vertically by a factor $1/\sqrt{3}$; remove the $2n$ odd vertices coming from the left boundary of the puzzle, together with the possible horizontal edges to their right; finally, double the length of the bottom vertical edges. Notice that edges $/$ and $\backslash$ of the original unit triangles are now sent to odd and even vertices of $\mathcal{T}_n$ respectively; see Figure~\ref{fig:fpltopuzzle}, right, for the end result.

\begin{figure}[!ht]
\begin{center}
\includegraphics[width=0.6\textwidth]{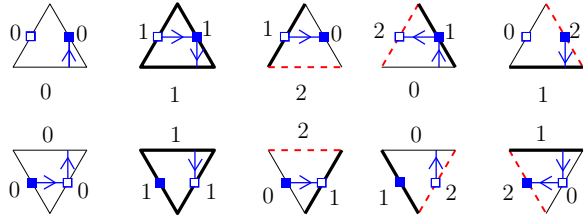}
\caption{ The local transformations in the main bijection.
  \label{fig:rulesoriented}}
\end{center}
\end{figure}

\begin{figure}[!ht]
 \begin{center}
 \includegraphics[width=\textwidth]{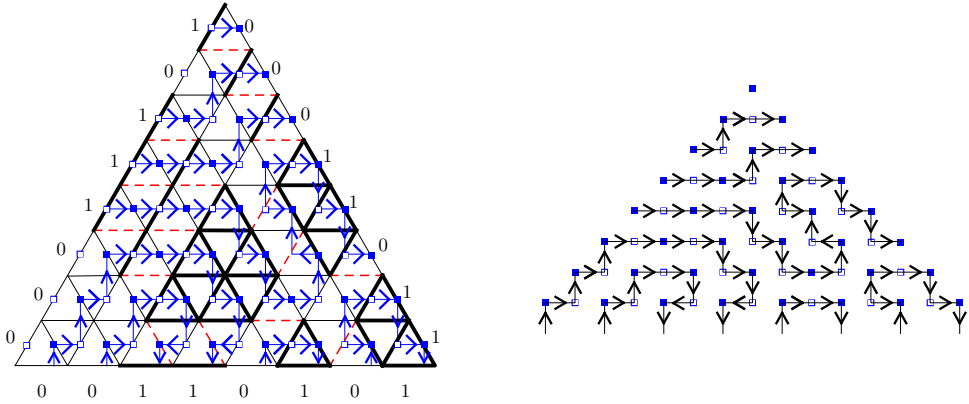}
 \caption{Example of the main bijection. \label{fig:fpltopuzzle}}
 \end{center}
 \end{figure}


  Looking at Figure~\ref{fig:rulesoriented}, one sees that, when two triangles of a puzzle share a common horizontal edge, the oriented vertical half edges drawn inside these triangles are compatible.  Indeed one notices that horizontal edges labeled $0,1$, and $2$ are respectively associated to oriented edges pointing up, pointing down, and to no such edge. Therefore the transformations produce a certain directed graph on $\mathcal{T}_n$. 

\begin{defi}[$\oPhi$ and $\Phi$]
\label{defi:phi}
  We define $\oPhi(P)$ to be the directed graph on $\mathcal{T}_n$ obtained by the construction above. We also define $\Phi(P)$ to be the underlying non-oriented graph.
\end{defi}

  We then have the following theorem:

\begin{thm}
\label{th:bij}
 Let $\si,\tau,\pi\in \Dn$ verify $\size(\si)+\size(\tau)=\size(\pi)$. Then $P\mapsto \Phi(P)$ is a bijection between $\KT$ and $\Tspt$.
\end{thm}

The proof is given in the next section. Notice that the definition of $\Phi$ does not need the introduction of $\oPhi$, as one can perform the transformation rules of Figure~\ref{fig:rulesoriented} while simply ignoring the arrows. But as we will see, the orientation is essential in the demonstration of the theorem.

\subsection{Proof of Theorem~\ref{th:bij}}
\label{sub:proof_bij}

We have the following properties:

\begin{prop}
\label{prop:proof} Let $\si,\tau,\pi\in\Dn$ be such that $\size(\si)+\size(\tau)=\size(\pi)$, and let $P$ be a puzzle in $\KT$. Then
\begin{enumerate}[(a)]
 \item $\oPhi(P)$ belongs to $\oTspt$;\label{it1}
 \item $\Phi(P)$ belongs to $\Tspt$;\label{it2}
 \item $\Phi$ is injective;\label{it3}
\end{enumerate}
\end{prop}

\begin{proof}
 
To prove~\eqref{it1} we check the conditions of Definition~\ref{defi:otfpls}. Let us first show that at each inner vertex of $\mathcal{T}_n$, the graph $\oPhi(P)$ has one incoming edge and one outgoing edge. For this one simply needs to check, on the local rules of Figure~\ref{fig:rulesoriented}, how the transformation acts on diagonal edges of the triangles. Looking at the three possible labels $0,1$, and $2$ and the two edges $/$,$\backslash$ of triangles, one checks that in each of these six cases the corresponding inner vertex has indeed one incoming edge and one outgoing edge. Now we need to verify that $\si,\tau,\pi$ are the boundary labels of the oriented graph configuration $\oPhi(P)$. In this case again, a direct inspection of the rules shows that this is indeed the case: it is immediate for $\pi$ and $\tau$, and for $\si$ one must remember that the vertices coming from the left side of $P$ must be deleted by the definition of $\oPhi$. In conclusion $\oPhi(P)$ is an oriented TFPL and~\eqref{it1} is proved. 

To prove~\eqref{it2}, it remains to verify verify that the bottom paths in $\Phi(P)$ follow the link pattern $\pi$: this is indeed the case thanks to  Proposition~\ref{prop:oTFPLs_to_TFPLs}\eqref{it3} and the second part of the following lemma:

\begin{lemma}
 \label{lem:noloops}
  $\oPhi(P)$ contains no (oriented) closed paths, and all Bottom paths in $\oPhi(P)$ are oriented from left to right.
\end{lemma}

\demof{Lemma~\ref{lem:noloops}} Suppose there is a path $p$ oriented from right to left in $\oPhi(P)$, and consider the largest $y$-coordinate $y_{max}$ visited by $p$. There is necessarily a horizontal edge $e$ on the line $y=y_{max}$, and $e$ must be oriented to the left by the (discrete) Jordan curve theorem. Now by observing the rules in Figure~\ref{fig:rulesoriented}, the edge $e$ is necessarily connected either to two edges going down or to two edges going up, and both cases contradict the definition of $y_{max}$, and therefore all Bottom paths are oriented from left to right.

Now suppose $\oPhi(P)$ has a closed loop $\ell$. If $\ell$ is oriented counterclockwise, then the previous analysis  leads to a contradiction. If $\ell$ is oriented clockwise, we consider the minimal $y$-coordinate visited by $\ell$; there must be a left oriented edge at this height, which leads to a contradiction here also.
\findem

 We now deal with~\eqref{it3}: notice first that $\oPhi$ is clearly injective. Indeed, in each local transformation of Figure~\ref{fig:rulesoriented}, the ten local configurations of oriented edges are distinct,  and thus determine uniquely the labeling of the unit triangle. To prove the injectivity of $\Phi$ comes down to proving that ``erasing the arrows'' in a configuration $\oPhi(P)$ is an injective operation. By Lemma~\ref{lem:noloops}, this is indeed the case: orienting all Bottom paths and Left-Right paths paths in $\Phi(P)$ from left to right gives back $\oPhi(P)$. This achieves the proof of the proposition.
\end{proof}

 \demof{Theorems~\ref{th:main} and~\ref{th:bij}} Let $\si,\tau,\pi\in \Dn$ be such that $\size(\si)+\size(\tau)=\size(\pi)$. By Proposition~\ref{prop:proof}, $\Phi$ is an injective mapping between $\KT$ and $\Tspt$. Therefore, we have by Theorem~\ref{th:KTLR} that
\[
 \cspt\leq\tspt.
\]

Now since $\frac{1}{H_{\si}H_{\tau}}>0$, Equation~\eqref{eq:identity_tc} expresses that a sum of nonnegative terms is equal to zero. Therefore all terms in this sum must be equal zero, which proves that $\cspt=\tspt$ and  achieves the proof of Theorem~\ref{th:main}. It also shows that $\Phi$ is a injection between two sets with the same cardinality, therefore $\Phi$ is a bijection, which is the content of Theorem~\ref{th:bij}. \findem

\subsection{Final comments}
We describe first a nice consequence of the bijection $\Phi$. Let the \emph{canonical orientation} of a TFPL $f$ be the oriented TFPL $\can(f)$ obtained by orienting all Bottom paths and Left-Right paths of $f$ from left to right, and all its Closed paths clockwise. If $f\in \Tspt$, then clearly $\can(f)\in \oTspt$. Say that a vertical (oriented) edge {\em even} if its lower vertex is even.

\begin{prop}
\label{prop:csq_bij}
 Let $\si,\tau,\pi\in\Dn$ and $f\in \Tspt$. Then $\size(\si)+\size(\tau)=\size(\pi)$ if and only if $\can(f)$ has neither consecutive left edges nor even vertical edge.
\end{prop}

\dem It is a byproduct of the proof of Theorem~\ref{th:bij} that, when $\size(\si)+\size(\tau)=\size(\pi)$, then $\can(f)$ has the form $\oPhi(P)$ or a puzzle $P\in\KT$. By inspection of the local transformations of Figure~\ref{fig:rulesoriented} applied to $P$, no two consecutive left edges can appear, and no even edge either, which proves the direct implication.

Consider now an element $f\in\oTspt$. We reverse the passage between triangular grid and square grid given in~\ref{sub:defi_phi}: we add one horizontal edge to the left of each Left-Right path, divide by $2$ the length of the bottom edges $e_i$, and rescale the TFPL vertically by a factor $\sqrt{3}$, so that one can superimpose / and $\backslash$ edges of the triangular lattice to odd and even vertices respectively. If $f$ avoids consecutive left edges and even vertical edges, then the fillings of the unit triangles belong to one of the ten possibilities of Figure~\ref{fig:rulesoriented}, and the boundary words of $f$ correspond to the labels from the same figure. Therefore, there exists a puzzle $P\in \KT $ such that $\oPhi(P)=f$. But such puzzles exist only if $\size(\si)+\size(\tau)= \size(\pi)$ by Theorem~\ref{th:KTLR}, which completes the proof.\findem

The preceding proposition thus gives a nice characterization of TFPLs. However the proof of this result is slightly unsatisfying from a combinatorialist's point of view, since one would like a direct proof of this fact by studying the structure of TFPL configurations. This would also give a direct proof of the bijectivity of $\Phi$, that is to say a proof which does not rely on Identity~\eqref{eq:identity_tc}.

In current work together with Ilse Fischer~\cite{TFPL3}, we achieve this by studying oriented TFPL configurations. For a given $k\geq 0$ and $\si,\tau,\pi$ such that $\size(\pi)-\left(\size(\si)+\size(\tau)\right)=k$, we are able to characterize elements $F$ of $\oTspt$ in terms of the occurrence of $k$ specific local sub-configurations in $F$. In the special case of $k=0$, we obtain the result of Proposition~\ref{prop:csq_bij} \emph{for all oriented TFPL configurations}, and not only the canonically oriented ones. It follows that all elements of $\oTspt$ are of the form $\Phi(P)$, from which one deduces easily that all maps in the following commutative diagram are bijections (the unlabeled map is simply the removal of the orientation):

\begin{center}
\includegraphics{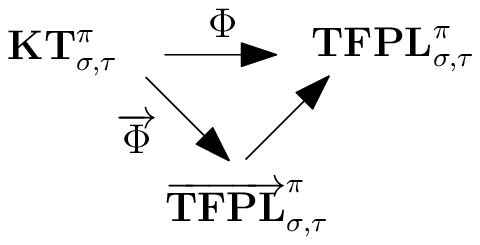}
\end{center}

%

\def\polhk\#1{\setbox0=\hbox{\#1}{{\o}oalign{\hidewidth
  \lower1.5ex\hbox{`}\hidewidth\crcr\unhbox0}}}

\end{document}